\documentclass[12pt]{amsart}

\usepackage{amsmath,amssymb,amsthm}
\usepackage[cal= cm, bb = ams, frak = euler, scr = rsfs]{mathalpha}
\usepackage{microtype}
\usepackage[T1]{fontenc}
\usepackage[utf8]{inputenc}
\usepackage{verbatim}
\usepackage{float}
\usepackage{cancel}
\usepackage{xcolor}
\usepackage[unicode=true,pdfusetitle,bookmarks=true,
bookmarksnumbered=false, breaklinks=false, backref=false,
colorlinks=true, linkcolor=blue!70!black, citecolor=black,
urlcolor=blue!78!red, pagebackref=false, final]{hyperref}
\usepackage{graphicx}
\usepackage{grffile}
\usepackage{array}
\usepackage[section]{placeins}
\usepackage[capitalise]{cleveref}

\newcommand{\clevertheorem}[3]{%
  \newtheorem{#1}[thm]{#2} \crefname{#1}{#2}{#3} }

\theoremstyle{plain} \newtheorem{thm}{Theorem}[section]
\crefname{thm}{Theorem}{Theorems} \newtheorem*{thm*}{Theorem}
\clevertheorem{prop}{Proposition}{Propositions}
\newtheorem*{prop*}{Proposition} \clevertheorem{lem}{Lemma}{Lemmas}
\clevertheorem{cor}{Corollary}{Corollaries}
\clevertheorem{conj}{Conjecture}{Conjectures}

\theoremstyle{definition}
\clevertheorem{defn}{Definition}{Definitions}
\clevertheorem{defns}{Definitions}{Definitions}
\clevertheorem{question}{Question}{Questions}
\clevertheorem{example}{Example}{Examples}
\clevertheorem{exercise}{Exercise}{Exercises}
\clevertheorem{notation}{Notation}{Notations}

\theoremstyle{remark} 
\clevertheorem{rk}{Remark}{Remarks} 

\usepackage[
obeyDraft, color=orange!80, 
bordercolor=black, textwidth=2cm, textsize=small, colorinlistoftodos]
{todonotes}

\makeatletter \providecommand\@dotsep{5} \makeatother

\newcommand{\itodo}[1]{\
  \todo[inline,color=gray!10,linecolor=gray!40!white,bordercolor=gray!30!white,size=\normalsize]{\textcolor{black!80!white}{#1}}}

\usepackage{ifdraft}

\ifdraft{
  \usepackage[top=.7in,bottom=.7in,left=1in,right=1in]{geometry}
  
  \usepackage[datesep={-},showseconds=false]{datetime2}
  \DTMsettimestyle{iso}
  
  \usepackage{draftwatermark} \SetWatermarkScale{1}
  \SetWatermarkAngle{90} \SetWatermarkLightness{.75}
    
  \SetWatermarkText{\shortstack{Draft : \today ~at~ \DTMcurrenttime
      \vspace{45pc}}} } {
  \usepackage[top=.7in,bottom=.7in,left=.8in,right=.8in]{geometry}
  
  \PassOptionsToPackage{disable}{todonotes} }

\usepackage{tikz} \usetikzlibrary{decorations.markings}
\usepackage{tikz-cd}

\hypersetup{pdfauthor={Nathan Fieldsteel},
  pdftitle={Buchsbaum-Eisenbud complexes of OI-modules},
  pdfkeywords={}, pdfsubject={}, pdfcreator={Emacs 27.2},
  pdflang={English}}

\author{Nathan Fieldsteel} \address{Department of Mathematics,
  University of Kentucky, Lexington, KY 40506 USA}
\email{nathan.fieldsteel@uky.edu}
\urladdr{\url{http://nathanfieldsteel.github.io}}

\author{Uwe Nagel} \address{Department of Mathematics, University of
  Kentucky, Lexington, KY 40506 USA} \email{uwe.nagel@uky.edu}
\urladdr{\url{http://www.ms.uky.edu/~uwenagel/}}

\title{Buchsbaum-Eisenbud complexes of OI-modules}

\newcommand{\N}{\mathbb{N}}

\newcommand{\A}{\mathbf{A}}
\newcommand{\F}{\mathbf{F}}
\newcommand{\bG}{\mathbf{G}}
\newcommand{\bI}{\mathbf{I}}
\newcommand{\M}{\mathbf{M}}
\newcommand{\bN}{\mathbf{N}}
\newcommand{\ffi}{\varphi}
\newcommand{\eps}{\varepsilon}

\newcommand{\Fo}[1]{{\mathbf{F}}^{\OI, #1}}

\DeclareMathOperator{\BE}{BE}
\DeclareMathOperator{\coker}{coker}
\DeclareMathOperator{\FI}{FI}
\DeclareMathOperator{\OI}{OI}
\DeclareMathOperator{\Hom}{Hom}
\DeclareMathOperator{\id}{id}
\DeclareMathOperator{\im}{im}

\begin{document}

\begin{abstract}
We extend the theory of Koszul and Buchsbaum-Eisenbud complexes to modules over commutative OI-algebras and show that they still have the familiar properties of the classical complexes. In particular, the OI-complexes are generically acyclic and often provide width-wise minimal free resolutions. 
\end{abstract}


\thanks{The second author was partially supported by Simons Foundation grant \#636513. }

\maketitle

\tableofcontents


\section{Introduction}\label{sec:introduction} 

The desire to study spaces of different dimensions simultaneously has motivated many concepts and results. For example, the theory of $\FI$-modules (over a fixed ring) allows one to express and establish results in representation stability, where $\FI$ is the category of finite sets with injections as morphisms. This theory has led to many further exciting developments (see, e.g.,  \cite{CEF, DEKL, SS-17}). In order to incorporate the study of Markov bases in algebraic statistics (see, e.g., \cite{HS}),   
$\FI$-modules over an $\FI$-algebra $\A$ were introduced in \cite{NR2}. Establishing finiteness results required even further 
generality by using the category $\OI$ of ordered finite sets with order-preserving injections as morphisms. Gr\"obner bases for $\OI$-modules over a polynomial $\OI$-algebra $\A$ were introduced in \cite{NR2}. A theory of Hilbert functions of graded $\OI$-modules has been initiated in \cite{NR1, N-hilb} (see also \cite{LNNR}). Over a noetherian polynomial $\OI$-algebra, the existence of resolutions by 
finitely generated free $\A$-modules was established in \cite{NR2}. However, very few examples of such resolutions or general constructions are known. 
This is not too surprising because even in the classical setting of commutative algebra rather few widely applicable methods exist. Arguably, the most famous complex in commutative algebra is the Koszul complex associated to an $R$-module homomorphism $F \to R$. More generally, there is a family of complexes associated to  an $R$-module homomorphism $F \to G$ where both $F$ and 
$G$ are finitely generated free $R$-modules. We call these complexes Buchsbaum-Eisenbud complexes because of their appearance in \cite{BE}. They include the Eagon-Northcott complex \cite{EN}. 

In this paper, we extend both constructions to the category of $\OI$-modules. 

\begin{thm}\label{intro:thm1}
  If \(\mathbf{A}\) is a commutative \(\OI\)-algebra and \(\varphi \colon \mathbf{F} \rightarrow \mathbf{A}\) is a morphism  from a finitely generated free \(\mathbf{A}\)-module to \(\mathbf{A}\), then  there exists a complex \(K_{\bullet}(\varphi)\) of finitely generated free \(\mathbf{A}\)-modules for which the module in homological degree \(d\) is the \(d^{\text{th}}\) exterior power of \(\mathbf{F}\), and in each width \(w\) the complex \(K_{\bullet}(\varphi)(w)\) is the ordinary Koszul complex on \(\varphi(w) \colon\mathbf{F}(w) \rightarrow \mathbf{A}(w)\). Furthermore, the complex \(K_{\bullet}(\varphi)\) is generically acyclic. 
 \end{thm}

We use the term ``generically acyclic'' to indicate that there are circumstances in which \(K_{\bullet}(\varphi)\) is acyclic.
 
\begin{thm}\label{intro:thm2}
  Let \(\mathbf{A}\) be a commutative \(\OI\)-algebra and let \(\varphi:\mathbf{F} \rightarrow \mathbf{G}\) be a morphism of finitely-generated free \(\mathbf{A}\)-modules where \(\mathbf{G}\) is generated in width \(0\). For each \(i \ge 0\), there exists a complex \(\BE^{i}_{\bullet}(\varphi)\) of finitely generated free \(\mathbf{A}\)-modules 
that in  each width \(w\) 
is a Buchsbaum-Eisenbud complex on \(\varphi(w)\). Furthermore, \(\BE^{i}_{\bullet}(\varphi)\) is generically acyclic. 
\end{thm}

Both of these constructions produce complexes of graded modules if $\ffi$ is a morphism of graded $\A$-modules. 

Minimality of graded free resolutions over a polynomial $\OI$-algebra was investigated in \cite{mincell}. In particular, it was shown 
that, if it exists,  a minimal graded free resolution of an $\A$ module is unique up to isomorphism and that, if $\A$ is noetherian, any $\A$-module has such a minimal resolution. 

Note that any free resolution $\F_{\bullet}$ of an $\A$-module $\M$ provides in each width $w$,  a free resolution of  $\M(w)$ as an $\A(w)$-module, namely  its width $w$ component $\F_{\bullet}(w)$. If $\F_{\bullet}(w)$ is a minimal free resolution for each $w$, the resolution $\F_{\bullet}$ is called \emph{width-wise minimal}. Such a resolution is necessarily minimal, but not every minimal resolution is width-wise minimal (see \cite{mincell}). Using cellular resolutions,  width-wise minimal free resolutions of some monomial ideals were constructed in \cite{mincell}. The complexes above provide many more such examples. If $\ffi$ is graded and $\BE^i_{\bullet} (\ffi)$ is acyclic, it gives a width-wise minimal free resolution. It resolves the \(i^{\text{th}}\) symmetric power of $\coker \ffi$ if $i \ge 1$. 

\begin{example}\label{runningexample:summary}
  Let \(\A\) be a polynomial \(\OI\)-algebra with three width \(1\) variables, so that in width \(w\) we have \(\A(w) = k[x_{i,j}\mid1\le i \le 3,~ 1\le j \le w]\). Let \(\mathbf{F}\) and \(\mathbf{G} = \A^3\) be free \(\A\)-modules, where \(\mathbf{F}\) has rank \(1\) and is generated in width \(1\), 
 and consider the morphism $\varphi \colon \mathbf{F} \rightarrow \mathbf{G}$ determined by sending the generator of $\F$ to 
 $ \begin{bmatrix}
        x_{1,1}\\
        x_{2,1}\\
        x_{3,1}
      \end{bmatrix} \in \bG(1)$. 
%
%
Thus, in each width \(w\), the map  \(\varphi(w)\) is represented by a generic $3 \times w$ matrix containing the variables of \(\A(w)\). If we let \(\mathbf{I}\) denote the ideal in \(\A\) generated in width \(3\) by the determinant of \(\varphi(3)\), then \(\mathbf{I}(w)\) is generated by the maximal minors of \(\varphi(w)\) for all \(w\). In this situation, $\BE^0_{\bullet} (\ffi)$ gives a  width-wise minimal graded resolution of \(\mathbf{I}\) by free \(\mathbf{A}\)-modules. Any width $w \ge 3$ component $\BE^0_{\bullet} (\ffi) (w)$ is the classical Eagon-Northcott complex to \(\varphi(w)\). 
\end{example}

Our constructions of the above complexes mirror the classical constructions, interpreted suitably. They use multilinear algebra and duals of modules.  Analogous complexes can be constructed in the category of $\FI$-modules. However, we are not aware of any (non-trivial) instance where these complexes provide width-wise minimal free resolutions. Thus, we focus on $\OI$-modules in this  
paper. 

In \Cref{prelim}, we recall basic definitions, in particular, of polynomial $\OI$-algebras and free modules.  We also discuss basic properties of tensor products and symmetric and exterior powers of modules over any commutative $\OI$-algebra. Furthermore, we introduce an internal hom functor that has the adjunction property (see \Cref{internalhom}). 
In \Cref{sec:classical BE}, we review the classical complexes. In particular, we write the classical Buchsbaum-Eisenbud 
complexes in a form so that they become amenable to extensions for $\A$-modules. To this end, we describe the differentials very explicitly. 
We conclude the preparations in \Cref{sec:free gives free}, where we show that tensor products and exterior and symmetric powers of finitely generated free $\A$-modules are free and finitely generated. Note that these statements witness identities 
involving binomial coefficients (see, e.g, \Cref{exa:tensor product}). We also characterize when a finitely generated free \(\mathbf{A}\)-module is isomorphic to its dual. Finally, our extensions of Koszul and Buchsbaum-Eisenbud complexes for $\OI$-modules are presented in Sections \ref{sec:OI-Koszul} and \ref{sec:OI-BE}. We illustrate them with several examples.


\section{Preliminaries}\label{prelim}

We are concerned with \(\OI\)-algebras over a commutative ring \(k\) and modules over these algebras. For the sake of completeness, we briefly recall the relevant definitions from \cite{NR2}, following \cite{mincell}. \itodo{more citations here}
We also introduce a suitable $\operatorname{Hom}$ functor and review multilinear constructions in the category of $\OI$-modules. 

\begin{defn}
  Let \(\OI\) be the category whose objects are totally-ordered finite sets and whose maps are order-preserving injective functions. An \(\OI\)\textit{-module} (over $k$) is a functor from \(\OI\) to the category of \(k\)-modules.
\end{defn}

The skeleton of $\OI$ is the subcategory of $\OI$ whose unique object of cardinality $w \in \N_0$ is $\{1,2,\ldots,w\}$. By abuse of notation we use \(w\) to refer to this object. We use the term \textit{width} to refer to this parameter. So for an 
\(\OI\)-module \(\mathbf{M}\), we call \(\textbf{M}(w)\) the \textit{width} \(w\) \textit{component of} \(\textbf{M}\).  By an 
\textit{element} of \(\mathbf{M}\) we mean an element \(m\) in some \(\textbf{M}(w)\), which we call a \textit{width} \(w\) \textit{element of} \(\textbf{M}\). \(\OI\)-modules form a functor category, where the morphisms between two $\OI$ modules $\M$ and $\bN$  are the natural transformations  \(\textbf{M} \to \textbf{N}\). The set of these is denoted \(\operatorname{Nat}(\mathbf{M}, \mathbf{N})\). 

A \textit{submodule} of an \(\OI\)-module \(\mathbf{M}\) is just a subfunctor of \(\mathbf{M}\) that is an \(\OI\)-module itself. Given a set \(G\) of elements of \(\mathbf{M}\), the \textit{submodule generated by} \(G\) is the smallest submodule that contains \(G\).

We are interested in a more general class of objects where, instead of looking at a functor from \(\OI\) to \(k\)-mod, we consider  families of modules over different rings where both the modules and the rings are parametrized by \(\OI\). 

\begin{defn}
  An \textit{\(\OI\)-algebra} is a covariant functor from \(\OI\) to the category of \(k\)-algebras. An \(\OI\)-algebra is called \textit{commutative}, \textit{graded} or \textit{graded-commutative} if it is a functor to the category of \(k\)-algebras with the corresponding property. In this paper we will work with commutative \(\OI\)-algebras unless otherwise stated, though we will typically drop the adjective. 
  
For an \(\OI\)-algebra \(\A\), an \(\A\)-module is an \(\OI\)-module \(\mathbf{M}\) together with an action $\mathbf{A} \otimes \mathbf{M} \stackrel{\mu}{\longrightarrow} \mathbf{M}$ that satisfies the usual commuting diagrams required of an action map (see \cite[Definition 3.1]{NR2}).
\end{defn}

Equivalently, an \(\OI\)-algebra \(\mathbf{A}\) is  a commutative monoid in the category of \(\OI\)-modules over $k$, and an \(\A\)-module is a module in the category-theoretic sense. We will not make use of this perspective in this paper.

Note that the constant functor from $\OI$ with image $k$ defines not only an \(\OI\)-module, but also an \(\OI\)-algebra. Modules over this \(\OI\)-algebra are exactly the ordinary \(\OI\)-modules defined above. So the theory of modules over an \(\OI\)-algebra is a generalization of the theory of \(\OI\)-modules.  From now on we will only consider modules over an \(\OI\)-algebra.


\begin{example}\label{runningexample:ideal}
  Consider the functor \(\mathbf{A} \colon \OI \rightarrow k\text{-alg}\) for which, in any width \(w\), \(\mathbf{A}(w)\) is the coordinate ring of the space of \(3 \times w\) matrices over \(k\), that  is, 
  \[
    \mathbf{A}(w) = k\begin{bmatrix}
      x_{11} & x_{12} & & x_{1w}\\
      x_{21} & x_{22} & \cdots & x_{2w}\\
      x_{31} & x_{32} & & x_{3w}
    \end{bmatrix}, 
  \]
  and where any \(\OI\)-morphism \(\varepsilon \colon u \rightarrow v\) induces an algebra map \( \varepsilon_* = \mathbf{A} (\varepsilon) \colon \mathbf{A}(u) \rightarrow \mathbf{A}(v)\) defined by \(x_{ij} \mapsto x_{i\varepsilon(j)}\). 
  In the language of \cite{NR2}, this is the polynomial \(\OI\)-algebra \(\left(\mathbf{X}^{\OI,1}\right)^{\otimes 3}\)

 Define a functor \(\mathbf{I} \colon \OI \rightarrow k\)-mod so that \(\mathbf{I}(w) \subseteq \mathbf{A}(w)\) is the ideal generated by the \(3\times 3\) minors of the above \(3 \times w\) matrix.  The maps \(\mathbf{A}(u) \rightarrow \mathbf{A}(v)\) induce maps \(\mathbf{I}(u) \rightarrow \mathbf{I}(v)\), and with these maps \(\mathbf{I}\) is an \(\mathbf{A}\)-module. As it is a submodule of $\mathbf{A}$, we call $\mathbf{I}$ an ideal. 
 It is generated as an \(\mathbf{A}\)-module by a single element in width \(3\), namely the determinant of the generic \(3\times 3\) matrix, as any \(3\times 3\) minor in any \(\mathbf{I}(w)\) with $w \ge 3$ is the image of this determinant under an appropriately chosen \(\OI\)-morphism.
  \end{example}
  
The most important \(\A\)-modules for our purposes are the free \(\A\)-modules, which we define now.

\begin{defn}
 For any integer $n \ge 0$, define an \(\mathbf{A}\)-module \(\mathbf{F}_{\mathbf{A}}^{\OI,n}\) as follows. For any $w \in \N_0$, its width $w$ component  \(\mathbf{F}_{\mathbf{A}}^{\OI,n}(w)\) is the free \(\mathbf{A}(w)\)-module with basis 
 $\left\{ e_{\nu} \mid \nu \in \operatorname{Hom}_{\OI}(n,w)\right\}$, and  any \(\OI\)-morphism \(\varepsilon \colon w \rightarrow w'\) acts on an element of \(\mathbf{F}_{\mathbf{A}}^{\OI,n}(w)\) by 
 \[
 \sum\limits_{\nu \in \operatorname{Hom}_{\OI}(n,w)} a_{\nu} e_{\nu} \mapsto \sum\limits_{\nu \in \operatorname{Hom}_{\OI}(n,w)} \varepsilon_* (a_{\nu}) e_{\varepsilon \circ \nu} \in \Fo{n}_{\A} (w').
 \] 
 \(\mathbf{F}_{\mathbf{A}}^{\OI,n}\) is generated as an \(\mathbf{A}\)-module by the element \(e_{\text{id}}\) in width \(n\). 
 
 A finitely generated \textit{free \(\mathbf{A}\)-module of rank $r$} is an $\mathbf{A}$-module that is isomorphic to a direct sum of the form 
 \[
 \mathbf{F}_{\mathbf{A}}^{\OI,n_{1}} \oplus \mathbf{F}_{\mathbf{A}}^{\OI,n_{2}} \oplus \cdots \oplus \mathbf{F}_{\mathbf{A}}^{\OI,n_{r}}. 
 \] 
\end{defn}

If $\A$ is clear from context, we often simply write $\Fo{w}$ instead of $\Fo{w}_{\A}$. 
By \cite[Lemma 3.6]{mincell}, if $\F$ is a finitely generated free $\A$-module, the number of elements in any basis and their widths are unique. In particular, the number of basis elements is equal to the rank of $\F$. 

A \textit{morphism of $\A$-modules} is a natural transformation \(\varphi \colon \mathbf{M} \rightarrow \mathbf{N}\) such that every  map  \(\varphi(w) \colon \mathbf{M}(w) \rightarrow \mathbf{N}(w)\) is a homomorphism of \(\mathbf{A}(w)\)-modules. 
Note that any morphism out of a free \(\mathbf{A}\)-module \(\mathbf{F}\) is uniquely determined by the images of the generators of \(\mathbf{F}\).  
 This justifies the name ``free.''

\begin{example}\label{runningexample:morphism}
  The map \(\varphi\) in \cref{runningexample:summary} is an example of a map between free \(\A\)-modules, where \(\mathbf{A}\) is the \(\OI\)-algebra considered in \cref{runningexample:ideal}
\end{example}

For any \(\OI\)-algebra \(\A\), the category of \(\A\)-modules has a pointwise direct sum and a pointwise tensor product defined by 
$\left(\mathbf{M} \oplus \mathbf{N}\right)(w) = \mathbf{M}(w) \oplus \mathbf{N}(w)$ and 
$\left (\mathbf{M}\otimes_{\mathbf{A}}\mathbf{N}\right)(w) = \mathbf{M}(w) \otimes_{\mathbf{A}(w)} \mathbf{N}(w)$, respectively. 
One checks that this width-wise definition of the tensor product agrees with the categorical definition of the tensor product as the coequalizer of the diagram 
\[
  \begin{tikzcd}
    \mathbf{M} \otimes \mathbf{A} \otimes \mathbf{N} \ar[r,shift left=.75ex]
    \ar[r,shift right=.75ex]
    &
    \mathbf{M} \otimes \mathbf{N}.
  \end{tikzcd}
\]

Later, we want to consider the dual of an $\A$-module. This requires some care and preparation. 
We use \(\operatorname{Nat}_{\A}(\mathbf{M}, \mathbf{N})\) to denote the set of \(\mathbf{A}\)-module morphisms \(\mathbf{M} \to \mathbf{N}\). We note that while this set is a \(k\)-module, it does not have an \(\A\)-module structure or even an \(\OI\)-module structure. So it is not the internal \(\operatorname{Hom}\) in the category of \(\A\)-modules. To define the latter, we use a different approach. 

By the adjoint functor theorem, the functor \({-}\otimes_{\A} \mathbf{N}\) has a right adjoint, which we will denote by \(\operatorname{Hom}_{\A}(\mathbf{N},{-})\), a functor from \(\mathbf{A}\)-mod to \(\mathbf{A}\)-mod. The enriched Yoneda lemma implies that in  every width \(w\) we have 
\[
\operatorname{Hom}_{\A}(\mathbf{N},{-})(w) = \operatorname{Nat}_{\A}(\mathbf{F}^{\OI,w}_{\A} \otimes_{\A} \mathbf{N}, {-}).
\] 
This suggests the following definition. 

\begin{defn}\label{internalhom}
For $\A$-modules $\M$ and $\bN$, define a functor $\operatorname{Hom}_{\A}(\mathbf{M}, \mathbf{N}) \colon \OI \to k$-mod: 
For any $w \in \N_0$, set 
\[
\operatorname{Hom}_{\A}(\mathbf{M}, \mathbf{N})(w) = \operatorname{Nat}_{\A}(\mathbf{F}^{\OI,w}_{\A}\otimes_{\A} \mathbf{M}, \mathbf{N}), 
\]
and, for any $\OI$-morphism \(\varepsilon \colon w \rightarrow w'\) and any $\A$-module morphism $\varphi \colon \mathbf{F}^{\OI,w}_{\A} \otimes_{\A} \mathbf{M} \rightarrow \mathbf{N}$, define a $k$-module homomorphism 
$\varepsilon_* \colon \operatorname{Nat}_{\A}(\mathbf{F}^{\OI,w}_{\A}\otimes_{\A} \mathbf{M}, \mathbf{N}) \to \operatorname{Nat}_{\A}(\mathbf{F}^{\OI,w'}_{\A}\otimes_{\A} \mathbf{M}, \mathbf{N})$ by putting 
\[
\big (\varepsilon_*(\varphi) (v) \big )\left(e_{\eta} \otimes m\right) = \varphi(e_{\eta \circ \varepsilon} \otimes m) \in \bN (v), 
\] 
where $v \in \N_0$, $\eta \in \Hom_{\OI} (w', v)$ and $m \in \M (v)$. 
\end{defn} 

\begin{prop}
Consider \(\A\)-modules  \(\mathbf{M}\) and \(\mathbf{N}\). For any integer $w \ge 0$, any $a \in \A(w)$ and any $\A$-module morphism $\ffi \colon \Fo{w} \otimes_{\A} \M \to \bN$, define an $\A$-module morphism $ a \ffi = a \cdot \ffi \colon \Fo{w} \otimes_{\A} \M \to \bN$ by mapping $e_{\mu} \otimes m \in \Fo{w}(v) \otimes_{\A(v)} \M(v)$ with $v \in \N_0$  (and so $\mu \in \Hom_{\OI} (w, v)$)
onto 
$\ffi \big (\mu_*(a) e_{\mu} \otimes m \big ) \in \bN (v)$. This gives  $\operatorname{Nat}_{\A}(\mathbf{F}^{\OI,w}_{\A} \otimes_{\A} \mathbf{M}, \bN)$ the structure of an $\A (w)$-module and endows $\operatorname{Hom}_{\A}(\mathbf{M}, \mathbf{N})$ with the structure of an $\A$-module. \
\end{prop}

\begin{proof}
  It is routine to check that the width-wise actions described above make  \(\operatorname{Hom}_{\A}(\mathbf{M},\mathbf{N})(w)\) an \(\A(w)\)-module. To verify that \(\operatorname{Hom}_{\A}(\mathbf{M}, \mathbf{N})\) is an \(\A\)-module, we also need to show that for any \(\OI\)-morphism \(\eta \colon v \rightarrow v'\), the diagram
  \[
    \begin{tikzcd}
      \A(v)  \otimes \operatorname{Hom}_{\A}(\mathbf{M}, \mathbf{N})(v) \ar[r] \ar[d] & \operatorname{Hom}_{\A}(\mathbf{M}, \mathbf{N})(v) \ar[d]\\
      \A(v') \otimes \operatorname{Hom}_{\A}(\mathbf{M}, \mathbf{N})(v') \ar[r] & \operatorname{Hom}_{\A}(\mathbf{M}, \mathbf{N})(v') 
    \end{tikzcd}
  \]
determined by \(\eta\) commutes. This means, for any $\A$-module morphism $\ffi \colon \Fo{v} \otimes_{\A} \M \to \bN$ and any $m \in \A(v)$, one has \(\eta_*(a\cdot \varphi) = \eta_*(a)\cdot \eta_*(\varphi)\). To check this is in any width $w$, consider any $\mu \in \Hom_{\OI} (v', w)$ and any $m \in \M(w)$. One computes
%
%
%
  \begin{align*}
    \eta_*(a \cdot \varphi)(e_\mu \otimes m) &= \big ( a \cdot \varphi \big ) (e_{\mu \circ \eta} \otimes m) \\
                                             &= \varphi \big ((\mu \circ \eta)_*(a) \cdot e_{\mu \circ \eta} \otimes m \big)\\
                                             &= \varphi \big (\mu_*(\eta_*(a)) \cdot e_{\mu \circ \eta} \otimes m \big)\\
                                             &= \eta_*(\varphi) \big (\mu_*(\eta_*(a)) \cdot e_\mu \otimes m \big )\\
                                             &= \big (\eta_*(a)\cdot \eta_* (\varphi) \big ) (e_\mu \otimes m),
  \end{align*}
as desired. 
\end{proof}

We always consider the $\A$-module structure on $\operatorname{Hom}_{\A}(\mathbf{M}, \mathbf{N})$ as defined above and refer to it as the \textit{internal hom} in the category of \(\A\)-modules. By its definition,  it   has the adjunction property 
\[
\operatorname{Nat}_{\A}(\mathbf{M} \otimes_{\A} \mathbf{N}, \mathbf{P}) = \operatorname{Nat}_{\A}(\mathbf{M},\operatorname{Hom}_{\mathbf{A}}(\mathbf{M}, \mathbf{P})). 
\]


We also need analogs of multilinear constructions in the category of \(\A\)-modules. 

\begin{defn}
  Let \(\mathbf{A}\) be a commutative OI-algebra over \(k\), and let \(\mathbf{M}\) be an \(\mathbf{A}\)-module. Define the \textit{tensor algebra} of \(\mathbf{M}\) over \(\mathbf{A}\) as 
 \[ 
 \mathbf{T}(\mathbf{M}) = \mathbf{A} \oplus (\mathbf{M}) \oplus (\mathbf{M} \otimes_{\mathbf{A}} \mathbf{M}) \oplus (\mathbf{M} \otimes_{\mathbf{A}} \mathbf{M} \otimes_{\mathbf{A}} \mathbf{M}) \oplus \cdots .
 \] 
\end{defn}

Note that in any width,  \( \big (\mathbf{T}(\mathbf{M}) \big )(w)\) is simply the tensor algebra of \(\mathbf{M}(w)\) as an \(\mathbf{A}(w)\)-module. This means that \(\mathbf{T}(\mathbf{M})\) is itself an \(\mathbf{A}\)-module, though it also comes equipped with an \(\mathbf{A}\)-linear multiplication map \(\mathbf{T}(\mathbf{M}) \otimes_{\mathbf{A}} \mathbf{T}(\mathbf{M}) \longrightarrow \mathbf{T}(\mathbf{M}) \) determined by the width-wise algebra structures.

\begin{defn} 
For any \(\mathbf{A}\)-module \(\mathbf{M}\), consider the two-sided ideal of the tensor algebra \(\mathbf{T}(\mathbf{M})\), 
 \[
  \mathbf{I} = \left\langle m \otimes n - n \otimes m \mid n, m ~\text{in}~ \mathbf{M}(w) ~\text{for some width}~ w \in \N_0\right\rangle. 
 \] 
 The quotient of \(\mathbf{T}(\mathbf{M})\) by \(\mathbf{I}\) is called the \textit{symmetric algebra} of \(\mathbf{M}\), denoted  \
 $\operatorname{Sym}_{\bullet}(\mathbf{M})$. 
\end{defn} 

 The symmetric algebra of \(\mathbf{M}\) is  graded, and we denote the components of its decomposition using the usual notation 
 \[
 \operatorname{Sym}_{\bullet}(\mathbf{M}) = \bigoplus\limits_{i = 0}^{\infty} \operatorname{Sym}_{i}(\mathbf{M}),
 \] 
 where each \(\operatorname{Sym}_{i}(\mathbf{M})\) is an \(\mathbf{A}\)-module. 

\begin{defn}
For any \(\mathbf{A}\)-module \(\mathbf{M}\), consider the two-sided ideal of the tensor algebra \(\mathbf{T}(\mathbf{M})\), 
 \[
 \mathbf{J} = \left\langle m \otimes m \mid m ~\text{in}~ \mathbf{M}(w) ~\text{for some width}~ w\right\rangle.
 \] 
 The quotient of \(\mathbf{T}(\mathbf{M})\) by \(\mathbf{J}\) is called the \textit{exterior algebra} of \(\mathbf{M}\), denoted 
 $\bigwedge\nolimits^{\bullet}(\mathbf{M})$. 
 \end{defn}
 
 The exterior algebra of \(\mathbf{M}\) is also graded, and we denote the graded components using the usual notation 
 \[
 \bigwedge\nolimits^{\bullet}(\mathbf{M}) = \bigoplus\limits_{i = 0}^{\infty} \bigwedge \nolimits^{i}(\mathbf{M}),
 \] 
 where each \(\bigwedge \nolimits^{i}(\mathbf{M})\) is an \(\mathbf{A}\)-module.

The exterior and the symmetric algebra of $\M$ could be defined point-wise. In fact, one checks that the width \(w\) component of the exterior algebra of \(\mathbf{M}\) over \(\mathbf{A}\) is simply the exterior algebra of \(\mathbf{M}(w)\) over \(\mathbf{A}(w)\), 
that is, 
\begin{equation*}
   \label{eq:components of exterior algebra}
 \left ( \bigwedge \nolimits^{\bullet}(\mathbf{M}) \right )(w) = \bigwedge\nolimits^{\bullet}\left( \mathbf{M}(w)\right).
\end{equation*}
Similarly, one has 
\[
 \left ( \operatorname{Sym}_{\bullet}(\mathbf{M}) \right )(w) = \bigoplus\limits_{i = 0}^{\infty} \operatorname{Sym}_{i}(\mathbf{M} (w)). 
 \]

\section{Classical Complexes Constructed via Multilinear Algebra}
     \label{sec:classical BE}

In this section we review classical complexes constructed by multilinear algebra, which we will generalize to the \(\OI\) setting in Sections \ref{sec:OI-Koszul} and \ref{sec:OI-BE}. In particular, we will discuss the Koszul complex and the Eagon-Northcott complex. The latter is part of a family of complexes we dub \textit{Buchsbaum-Eisenbud complexes} because of their appearance in \cite{BE}. 
It is important to note that our treatment of these complexes differs slightly from their presentations in other sources (see, e.g., \cite{BV, E-book}). The differences in our presentation do not meaningfully change the complexes in the classical setting, but allow us to obtain analogs in the category of $\OI$-modules in subsequent sections. With this goal in mind, we describe all chain maps in detail. 


We begin with the Koszul complex, which is a complex of free \(R\)-modules determined by the data of a map from a free \(R\)-module to \(R\).

\begin{defn}
  Let \(R\) be a commutative ring, let \(F\) be a finitely generated free \(R\)-module of rank \(n\), and let \(\varphi: F \rightarrow R\) be an \(R\)-module map. We define \emph{the Koszul complex} \(K_{\bullet}(\varphi)\) to be the chain complex with modules $K_{d}(\varphi) = \bigwedge\nolimits^{d}(F)$ for \(d \ge 0\), where for \(d \ge 1\) the map \(K_{d}(\varphi) \rightarrow K_{d-1}(\varphi)\) is defined on any wedge of \(d\) elements of \(F\) by 
 \[
 e_{1}\wedge e_{2} \wedge \cdots \wedge e_{d} \mapsto \sum\limits_{j = 1}^{d} (-1)^{j+1} \varphi(e_{j}) e_{1}\wedge e_{2}\wedge \cdots \wedge \cancel{e_{j}} \wedge \cdots \wedge e_{n}.
 \]
\end{defn}

The Koszul complex is exact if and only if the elements \(\varphi(e_{1}), \ldots, \varphi(e_{n})\) form a regular sequence. In this case, if $R$ is a standard graded algebra over a field it provides a graded minimal free resolution of the ideal generated by the image of \(\varphi\). 
For example, if \(R\) is a noetherian polynomial ring, and \(\varphi: F \rightarrow R\) sends the basis of \(F\) to the variables of \(R\), the Koszul complex \(K_\bullet(\varphi)\) provides a minimal free resolution of the graded irrelevant ideal of \(R\). 

The family of Buchsbaum-Eisenbud complexes provides a significant generalization of the Koszul complex. Their input is a homomorphism of free \(R\)-modules.

\begin{defn}
    \label{def:classical BE complex}
  Let \(F\) and \(G\) be free \(R\)-modules with bases \(\{f_{1}, \ldots, f_{n}\}\) and \(\{g_{1}, \ldots, g_{r}\}\), and let \(\varphi \colon F \rightarrow G\) be an \(R\)-module map, represented by a matrix \(A = \left(a_{i,j}\right) \in R^{r \times n}\) with respect to these bases. 
Let \(({-})^{*}\) be the usual dualizing functor \(\operatorname{Hom}_{R}({-},R)\) and let \(\{f_{1}^{*}, \ldots, f_{n}^{*}\}\) be the usual 
 dual basis for \(F^{*}\). 
 For any integer $i \ge 0$, we define the $i^{\text{th}}$ Buchsbaum-Eisenbud complex to $\ffi$ as the chain complex \(\text{BE}^i_{\bullet} (\varphi)\), 
 \begin{align*}
 0 & \to \bigwedge\nolimits^{r}(G^{*})\otimes S_{n-r-i}(G^{*})\otimes \bigwedge \nolimits^{n} F \to \bigwedge\nolimits^{r}(G^{*})\otimes S_{n-r-i-1}(G^{*})\otimes \bigwedge \nolimits^{n-1} F \to  \cdots \\
 & 
 \to \bigwedge\nolimits^{r}(G^{*})\otimes S_{0}(G^{*})\otimes \bigwedge \nolimits^{r+i} F  \stackrel{\alpha_i}{\longrightarrow}  \bigwedge\nolimits^{i} F    \otimes S_{0}(G) \to  \bigwedge\nolimits^{i-1} F \otimes S_{1}(G) \to \cdots \\
 & \to \bigwedge\nolimits^{1} F \otimes S_{i-1}(G) \to \bigwedge\nolimits^{0} F \otimes S_{i}(G) \to 0,  
 \end{align*}
 where $S_i$ is short for $\operatorname{Sym}_i$ and the map $\alpha_i \colon \bigwedge\nolimits^{r}(G^*)\otimes S_{0}(G^{*})\otimes \bigwedge\nolimits^{r+i} F \rightarrow \bigwedge\nolimits^{i} F \otimes S_{0}(G)$ is defined by 
 \[
 g_{1}^{*}\wedge \cdots \wedge g_{r}^{*} \otimes f_{j_{1}}\wedge \cdots\wedge f_{j_{r+i}}\mapsto \sum\limits_{\substack{I \subseteq J\\|I| = r}}\operatorname{sgn}(I\subseteq J)\cdot\operatorname{det}(A_{I}) f_{J \setminus I}. 
\] 
Here, $J = \{j_1 < \ldots < j_{r+i}\}$, \ $f_J = f_{j_{1}}\wedge \cdots\wedge f_{j_{r+i}}$, \ $f_{J \setminus I}$ is the wedge of elements with indices in $J \setminus I$, \ $A_I$ is the $r$-minor of $A$ with column indices in $I$ and $\operatorname{sgn}(I\subseteq J)$ is the sign of the permutation 
$
\begin{pmatrix}
j_1 & \ldots & j_r & j_{r+1} & \ldots & j_{r + i} \\
i_1 & \ldots & i_r & k_1 & \ldots & k_i
\end{pmatrix}
$
with $I = \{i_1 < \cdots < i_r\}$ and $J \setminus I = \{k_1 < \cdots < k_i\}$. 
By definition, if $i+r >n$ the length of $\BE^i_{\bullet} (\ffi)$ is $\min \{i, n\}$. 

The part of \(\BE^i_{\bullet} (\varphi)\) to the right of $\alpha_i$ is a part of the Koszul complex determined by the $S (G)$-linear map 
$F \otimes S(G) \to S(G)$ induced by $\ffi$. Concretely, the map 
$
\bigwedge\nolimits^{p}(F) \otimes S_{q}(G) \rightarrow \bigwedge\nolimits^{p-1}(F) \otimes S_{q+1}(G)
$
is defined by 
\[
f_{i_{1}}\wedge \cdots \wedge f_{i_{p}} \otimes g_{j_{1}} \cdots g_{j_{q}} \mapsto \sum\limits_{\ell = 1}^{p}\left( f_{i_{1}}\wedge \cdots \wedge \cancel{f_{i_{\ell}}} \wedge \cdots \wedge f_{i_{p}} \otimes \varphi(f_{i_{\ell}}) g_{j_{1}} \cdots g_{j_{q}}\right). 
\]
Any map 
$
\bigwedge\nolimits^{r}(G^{*})\otimes S_{p}(G^{*})\otimes\bigwedge\nolimits^{q} F \rightarrow \bigwedge\nolimits^{r}(G^{*})\otimes S_{p-1}(G^{*})\otimes\bigwedge\nolimits^{q-1} F
$ 
to the left of $\alpha_i$ is defined by  
\[
g_{i_{1}}^{*}\cdots g_{i_{p}}^{*} \otimes f_{j_{1}}\wedge \cdots \wedge f_{j_{q}} \mapsto \sum_{i \in \{i_{1}, \ldots, i_{p}\}}\Big(g_{i}(g_{i_{1}}^{*}\cdots g_{i_{p}}^{*}) \otimes \sum\limits_{\ell = 1}^{q}(-1)^{\ell+1} a_{i,j_{\ell}} f_{j_{1}}\wedge \cdots \wedge \cancel{f_{j_{\ell}}} \wedge \cdots \wedge f_{j_{q}}\Big). 
\] 
Here,  \(g_{i}(g_{i_{1}}^{*}\cdots g_{i_{p}}^{*})\) denotes the divided power action of \(S_{\bullet}(G)\) on \(S_{\bullet}(G^{*}).\)
\end{defn}

Often, the part of \(\BE^i_{\bullet} (\varphi)\) to the left of the map $\alpha_i$ is given as the dual of a Koszul complex used for the beginning of \(\text{BE}^i_{\bullet} (\varphi)\) (see, e.g., \cite{BV}). Applying the explicit isomorphisms given in \cite[Section 1.6]{BH}, one checks that one obtains the maps as given above. The description of the above maps is not coordinate free, but turns out useful in what follows.

The factor $\bigwedge\nolimits^{r}(G^{*})$ appearing in terms to the left of \(\alpha_i\) in \(\BE^i_{\bullet} (\varphi)\) is often omitted. 
If $R$ is a standard graded algebra, we assume that $\ffi$ is a map of graded $R$-modules, i.e., it has degree zero. In this case, the factor $\bigwedge\nolimits^{r}(G^{*})$ is needed to turn \(\BE^i_{\bullet} (\varphi)\) into a complex of graded free $R$-modules. 



It is well-known that the complexes \(\text{BE}^i_{\bullet} (\varphi)\) are generically exact (see, e.g., \cite[Theorem A2.10]{E-book}). 
Note that the arguments also work if \(i > n-r+1\) or if $n < r$. 

\begin{example}
For  a generic $3 \times w$ matrix  $B$, denote by $I_3$ the ideal generated by the $3$-minors of $B$. Using the notation of 
\Cref{runningexample:summary}, if $w \ge 3$ then the quotient $\A (w)/\bI_3 (A)$ is resolved as an $\A (w)$-module by \(\text{BE}^0_\bullet (\varphi)\), where $\ffi \colon (\A (w))^w \to (\A (w))^3$ is the map given by $B$. This was shown first by Eagon and Northcott. Often, the complex \(\text{BE}^0_{\bullet} (\varphi)\) is called an Eagon-Northcott complex because of \cite{EN}. 
\end{example}

\section{Exterior and Symmetric Powers and Duals of Free $\OI$-modules} 
\label{sec:free gives free}

Our goal is to generalize the complexes from the previous section to the setting of modules over \(\OI\)-algebras. In this section, we establish some needed preparations. Most of the following statements are familiar classical results with commutative rings replaced by commutative \(\OI\)-algebras. However, these similarities can mask important differences.  For example,  the tensor product of two free  $\A$-modules of rank \(1\) can be very complicated. Furthermore, duality behaves poorly in many cases. 

We begin with a criterion for checking whether an \(\A\)-module is free.

\begin{lem}\label{lem:freecriterion}
  Let \(\mathbf{A}\) be a commutative OI-algebra, and let \(\mathbf{M}\) be a finitely-generated \(\mathbf{A}\)-module. \(\mathbf{M}\) is a free \(\mathbf{A}\)-module if and only if \(\mathbf{M}\) has a {generating set} \(G = \{m_{1}, \ldots, m_{n}\}\) satisfying the following two conditions: 

  \begin{enumerate}

  \item [(i)] For each \(m_{i}\) in \(G\), the \(\mathbf{A}\)-{submodule} of \(\mathbf{M}\) generated by \(m_{i}\) is a free \(\mathbf{A}\)-module of rank 1.
    
  \item [(ii)] For each \(m_{i}\) in \(G\), the {intersection} of the \(\mathbf{A}\)-submodule of \(\mathbf{M}\) generated by \(m_{i}\) and the \(\mathbf{A}\)-submodule of \(\mathbf{M}\) generated by \(G \setminus \{m_{i}\}\) is \(0\).
    
  \end{enumerate}
\end{lem}

\begin{proof}
  It is immediate that for a free \(\mathbf{A}\) module \(\mathbf{F}\), any basis will satisfy these two conditions. Conversely, let \(\mathbf{M}\) be a finitely generated \(\mathbf{A}\)-module, and let \(G\) be a generating set satisfying both conditions. We will prove that \(\mathbf{M}\) is free by induction on \(|G|\).

  If \(|G| = 1\), condition (i) implies that \(\mathbf{M}\) is free.  If \(|G| > 1\), then let \(\mathbf{M}'\) be the free submodule generated by \(m_{1}\), and let \(\mathbf{M}''\) be the submodule generated by \(G \setminus \{m_{1}\}\). The set \(G \setminus \{m_{1}\}\) satisfies the two conditions from the lemma, so by induction \(\mathbf{M}''\) is free. We can combine the submodule inclusion maps \(\mathbf{M}' \hookrightarrow \mathbf{M}\) and \(\mathbf{M}'' \hookrightarrow \mathbf{M}\) into a map on the direct sum \(\mathbf{M}'\oplus\mathbf{M}'' \longrightarrow \mathbf{M}\), yielding a map which is surjective since \(G\) is a generating set for \(\mathbf{M}\), and injective because \(\mathbf{M}' \cap \mathbf{M}''\) is zero. This completes the proof.
\end{proof}

The first use we find for this lemma is to prove that the tensor product of free \(\A\)-modules is free.

\begin{prop}\label{tensorfree}
  Let \(\mathbf{A}\) be a commutative $\OI$-algebra. Given two finitely generated free \(\mathbf{A}\)-modules \(\mathbf{F}\) and \(\mathbf{F}'\), their tensor product \(\mathbf{F} \otimes_{\mathbf{A}} \mathbf{F}'\) is also free.
\end{prop}

\begin{proof}
Note that any tensor product of free \(\mathbf{A}\)-modules of finite rank can be written as a direct sum of tensor products of rank 1 \(\mathbf{A}\)-modules. It thus suffices to prove the assertion in the case where \(\mathbf{F}\) and \(\mathbf{F}'\) have rank 1. 

  
  To prove the proposition in the rank \(1\) case, we explicitly describe a generating set for \(\M = \mathbf{F}^{\text{OI},u} \otimes_{\mathbf{A}} \mathbf{F}^{\text{OI},v}\) in terms of \(u\) and \(v\) and show that it satisfies the two conditions in \cref{lem:freecriterion}. In a fixed width \(w\), the module \(\left(\mathbf{F}^{\text{OI},u} \otimes_{\mathbf{A}} \mathbf{F}^{\text{OI},v}\right)(w)\) is a free \(\mathbf{A}(w)\)-module with an \(\mathbf{A}(w)\)-module basis given by the set of pure tensors of basis elements of \(\mathbf{F}^{\text{OI,u}}(w)\) and \(\mathbf{F}^{\text{OI,v}}(w)\), i.e. it has a basis of the form \[\left\{ e_{\mu} \otimes e_{\nu}\mid \mu \in \operatorname{Hom}_{\text{OI}}(u,w)~\text{and}~ \nu \in \operatorname{Hom}_{\text{OI}}(v,w)\right\}.\]

  For a particular pure tensor \(e_{\mu} \otimes e_{\nu}\) in this set, we will use \(I_{\mu,\nu}\) to denote the union of the images of the indexing maps \[I_{\mu,\nu} = \operatorname{im}(\mu) \cup \operatorname{im}(\nu),\] which we think of as a subset of \([w] = \{1,2,\ldots, w\}\). Let \(G\) be the set of pure tensors in $\M$ for which the union of these images is the full set \([w]\), i.e. 
\[
G = \{e_{\mu} \otimes e_{\nu} \in \M(w) \mid w \in \N_0 \text{ and }  \ I_{\mu,\nu} = [w] \}.
\]

  We claim that \(G\) satisfies the conditions of \cref{lem:freecriterion}. To prove that \(G\) is a generating set, let \(e_{\mu} \otimes e_{\nu}\) be an arbitrary pure tensor in width \(w\), with \(|I_{\mu,\nu}| = w' \le w\). There is a unique OI morphism \(\varepsilon \in \operatorname{Hom}_{\text{OI}}(w',w)\) whose image is \(I_{\mu,\nu}\), namely the inclusion from \([w']\) to \(I_{\mu,\nu}\). If we define \(\mu' \in \operatorname{Hom}_{\text{OI}}(u,w')\) by \(\mu'(i) = \varepsilon^{-1}(\mu(i))\), and \(\nu' \in \operatorname{Hom}_{\text{OI}}(v,w')\) by \(\nu'(i) = \varepsilon^{-1}(\nu(i))\), we see that \(e_{\mu'}\otimes e_{\nu'}\) is in \(G\), and \(e_{\mu} \otimes e_{\nu}\) is in the \(\mathbf{A}\)-submodule generated by \(e_{\mu'}\otimes e_{\nu'}\). 
  
  To verify condition (i), let \(e_{\mu} \otimes e_{\nu}\) be any width \(w\) element of \(G\). Define a morphism $\mathbf{F}^{\text{OI},w} \to \langle e_{\mu} \otimes e_{\nu} \rangle$ by mapping $e_{\varepsilon} \in \mathbf{F}^{\text{OI},w} (w')$ onto $e_{\varepsilon \circ \mu} \otimes e_{\varepsilon \circ \nu}$. One checks that this map is an isomorphism of $\A$-modules.

 To prove that condition (ii) holds, it is enough to show,  for any two distinct elements $e_{\mu} \otimes e_{\nu} \in \M(w)$ and 
 $e_{\mu'} \otimes e_{\nu'} \in \M(w')$ in $G$ and any two maps $\varepsilon \in \Hom_{\OI} (w, \tilde{w})$, $\varepsilon' \in \Hom_{\OI} (w', \tilde{w})$, one has $(\varepsilon \circ \mu, \varepsilon \circ \nu) \neq (\varepsilon' \circ \mu', \varepsilon' \circ \nu')$. Indeed, using that 
 $w = |I_{\mu,\nu}| = | I_{\varepsilon \circ \mu, \varepsilon \circ \nu}|$ and the analogous fact for the other element, this is not too difficult to show. 
\end{proof}

Observe that the generators of $\mathbf{F}^{\text{OI},u} \otimes_{\mathbf{A}} \mathbf{F}^{\text{OI},v}$ have any width $w$ with 
$\max \{u, v\} \le w \le u+v$. 

\begin{example} 
   \label{exa:tensor product}
Consider the free rank 1 \(\mathbf{A}\)-modules \(= \mathbf{F}^{\OI,2}_{\mathbf{A}}\) and \(\mathbf{F}^{\OI,3}_{\mathbf{A}}\).  
Their tensor product has a basis with 3, 12 and 10 elements of width 3, 4 and 5, respectively. Thus, there is an isomorphism 
\[
\mathbf{F}^{\OI,2}_{\mathbf{A}} \otimes_{\A} \mathbf{F}^{\OI,3}_{\mathbf{A}} \cong (\Fo{3})^3 \oplus (\Fo{4})^{12} \oplus (\Fo{5})^{10}. 
\]
Comparing ranks in width $w$, we obtain for any $w \in \N$ the identity, 
\[
\binom{w}{2} \cdot \binom{w}{3} = 3 \binom{w}{3} + 12 \binom{w}{4} + 10 \binom{w}{5}. 
\]
\end{example}

Turning to exterior powers, we show that a familiar formula in the classical setting remains true for $\A$-modules. 

\begin{lem}\label{powersum}
  If \(\mathbf{M}\) and \(\mathbf{N}\) are \(\mathbf{A}\)-modules, then one has for every integer $i \ge 0$, 
\[\bigwedge\nolimits^{i}\left(\mathbf{M} \oplus \mathbf{N}\right) \cong \bigoplus\limits_{j = 0}^{i}\left(\bigwedge\nolimits^{i-j}\mathbf{M}\right) \otimes_{\mathbf{A}} \left(\bigwedge\nolimits^{j}\mathbf{N}\right). 
\]
\end{lem}

\begin{proof}
By Equation \eqref{eq:components of exterior algebra}, 
we have in each width \(w\) a  classical isomorphism of $\A(w)$-modules 
\begin{align*}
\left(\bigwedge\nolimits^{i}\left(\mathbf{M} \oplus \mathbf{N}\right)\right)(w) =  
\bigwedge\nolimits^{i} \big ( \M(w) \oplus \bN (w) \big ) 
\stackrel{\varphi}{\rightarrow} &
\bigoplus\limits_{j = 0}^{i} \bigwedge\nolimits^{i-j} (\M(w)) \otimes_{\A(w)} \bigwedge\nolimits^{j} (\bN(w)) 
= \\
& \Big(\bigoplus\limits_{j = 0}^{i} \big (\bigwedge\nolimits^{i-j}\mathbf{M} \big) \otimes_{\mathbf{A}} \big (\bigwedge\nolimits^{j}\mathbf{N} \big)\Big)(w). 
\end{align*} 
Thus, it is enough to verify that for any \(\OI\)-morphism \(\varepsilon \colon w \to w'\) the following diagram is commutative:
  \[
    \begin{tikzcd}[row sep=0.5em]
      \left (\bigwedge\nolimits^{i}(\mathbf{M} \oplus \mathbf{N}) \right ) (w) \ar[rr,"\varphi"] \ar[dd,"\varepsilon_*"] && \Big ( \bigoplus\limits_{j = 0}^{i} \big (\bigwedge\nolimits^{i-j}\mathbf{M} \big) \otimes_{\mathbf{A}} \big (\bigwedge\nolimits^{j}\mathbf{N} \big ) \Big)(w) \ar[dd,"\varepsilon_*"]\\
      && \\
      \left(\bigwedge\nolimits^{i}(\mathbf{M} \oplus \mathbf{N})\right) (w') \ar[rr,"\varphi"]&& 
      \Big ( \bigoplus\limits_{j = 0}^{i} \big (\bigwedge\nolimits^{i-j}\mathbf{M} \big) \otimes_{\mathbf{A}} \big (\bigwedge\nolimits^{j}\mathbf{N} \big ) \Big)(w')
    \end{tikzcd}
  \]
This is straightforward.
\end{proof}

Next, we show that the exterior power of a free \(\A\)-module is still free.

\begin{prop}\label{exteriorfree}
  Let \(\mathbf{A}\) be a commutative $\OI$-algebra, and let \(\mathbf{F}\) be a finitely-generated free \(\mathbf{A}\)-module. Then for every integer \(i \ge 0\), the exterior power $\bigwedge\nolimits^{i} \mathbf{F}$ is also a finitely-generated free \(\mathbf{A}\)-module.
\end{prop}

\begin{proof}
  Because of \cref{powersum}, it suffices to prove the statement in the case where \(\mathbf{F} = \mathbf{F}^{\text{OI},d}\) is a free \(\mathbf{A}\)-module of rank \(1\). As for \cref{tensorfree}, the proof proceeds by giving an explicit description of a generating set \(G\) for \(\bigwedge \nolimits^{i}\mathbf{F}^{\text{OI},d}\) which satisfies the conditions in \cref{lem:freecriterion}.

  In any given width \(w\), using the standard lexicographic ordering on the set of maps \(\operatorname{Hom}_{\text{OI}}(d,w)\), we have a basis \(B(w)\) for the free \(\mathbf{A}(w)\)-module \(\left(\bigwedge \nolimits^{i}\mathbf{F}^{\text{OI},d}\right)(w)\) given by 
 \[  
  B(w) = \left\{ e_{\varepsilon_{1}} \wedge \ldots \wedge e_{\varepsilon_{i}} ~| ~ \varepsilon_{1}, \ldots, \varepsilon_{i} \in \operatorname{Hom}_{\text{OI}}(d,w) ~\text{and}~\varepsilon_{1} < \ldots < \varepsilon_{i} \right\}.
  \]
 For such a basis element \(e_{\varepsilon_{1}} \wedge \ldots \wedge e_{\varepsilon_{i}}\), consider the union of the images of these indexing maps, which we denote by \(I_{\varepsilon_{1}, \ldots, \varepsilon_{i}}\). Let \(G\) be the following subset of the union of the bases $B(w)$:
\[
G = \{e_{\varepsilon_{1}} \wedge \ldots \wedge e_{\varepsilon_{i}} \mid e_{\varepsilon_{1}} \wedge \ldots \wedge e_{\varepsilon_{i}} \in B(w) ~\text{for some}~w \in \N_0~\text{and}~ I_{\varepsilon_{1}, \ldots, \varepsilon_{i}} = [w]\}.
\]
  
To show that \(G\) generates  $\bigwedge\nolimits^{i} \mathbf{F}$, it is enough to establish that any \(e_{1} \wedge \ldots \wedge e_{i}\) in any \(B(w)\) is in the \(\mathbf{A}\)-submodule of \(\bigwedge \nolimits^{i} \mathbf{F}\) generated by \(G.\) 
If \(|I_{\varepsilon_{1}, \ldots, \varepsilon_{i}}| = w\) the element is in \(G\), and we are done. 
If \(|I_{\varepsilon_{1}, \ldots, \varepsilon_{i}}| = w' < w\), let $\varepsilon$ be the unique $\OI$-morphism \(\varepsilon\) from \([w']\) to \([w]\) whose image is \(I_{\varepsilon_{1}, \ldots, \varepsilon_{i}}\). 
Defining \(\varepsilon_{j}' \in \operatorname{Hom}_{\OI}(d,w')\) 
on any \(k \in [d]\) by \(\varepsilon^{-1}(\varepsilon_{j}(k))\),  
we see that the element $e_{\varepsilon_{1}'} \wedge \ldots \wedge e_{\varepsilon_{i}'}$ is in \(G\) and that the induced map 
\(\varepsilon \colon \bigwedge\nolimits^{i}\mathbf{F}(w') \rightarrow \bigwedge\nolimits^{i}\mathbf{F}(w)\)  sends this element to \(e_{\varepsilon_{1}} \wedge \ldots \wedge e_{\varepsilon_{i}}\). So \(G\) is a generating set for \(\bigwedge\nolimits^{i} \mathbf{F}\) as an \(\mathbf{A}\)-module.

  To verify that $G$ satisfies Condition (i) of \Cref{lem:freecriterion}, consider any element \(e_{\varepsilon_{1}} \wedge \ldots \wedge e_{\varepsilon_{i}} \in B(w)\) of \(G\).  Define a morphism $\Fo{w} \to \langle e_{\varepsilon_{1}} \wedge \ldots \wedge e_{\varepsilon_{i}} \rangle$ by mapping $e_{\varepsilon} \in \mathbf{F}^{\text{OI},w} (w')$ onto $e_{\varepsilon\circ\varepsilon_{1}} \wedge \ldots \wedge e_{\varepsilon \circ \varepsilon_{i}}$. One checks that this map is an isomorphism of $\A$-modules.   
  
 To prove that $G$ satisfies Condition (ii) of \Cref{lem:freecriterion}, it is enough to show for any two distinct elements 
 $e_{\varepsilon_{1}} \wedge \ldots \wedge e_{\varepsilon_{i}} \in B(w)$ and 
 $e_{\varepsilon'_{1}} \wedge \ldots \wedge e_{\varepsilon'_{i}} \in B(w')$ of $G$ and any two maps $\varepsilon \in \Hom_{\OI} (w, \tilde{w})$, $\varepsilon' \in \Hom_{\OI} (w', \tilde{w})$,  one has 
 $(\varepsilon \circ \varepsilon_{1}, \ldots, \varepsilon \circ \varepsilon_{i}) \neq (\varepsilon' \circ \varepsilon'_{1}, \ldots, \varepsilon' \circ \varepsilon'_{i})$. As in the proof of \Cref{tensorfree}, this follows by using $|I_{\varepsilon \circ \varepsilon_{1}, \ldots, \varepsilon \circ \varepsilon_{i}}| = |I_{\varepsilon_{1}, \ldots, \varepsilon_{i}}|$. 
\end{proof}

Observe that the generators of $\bigwedge\nolimits^{i} \Fo{d}$ have any width $w$ with 
$w_0 \le w \le d i$, where $w_0 = \min \{ w \in \N_0 \mid \binom{w}{d} \ge i\}$.  

\begin{example}
The $\A$-module $\bigwedge^2 \Fo{3}$ has rank 31. It 
has a basis with 6, 15 and 10 elements of width 4, 5 and 6, respectively. Thus, there is an isomorphism 
\[
\bigwedge^2 \Fo{3} \cong (\Fo{4})^6 \oplus (\Fo{5})^{15} \oplus (\Fo{6})^{10}. 
\]
Comparing ranks in width $w$, we obtain for any $w \in \N$ the identity
\[
\binom{\binom{w}{3}}{2}  = 6 \binom{w}{4} + 15 \binom{w}{5} + 10 \binom{w}{6}. 
\]
\end{example}

Now we establish the analogous result for the symmetric power of a free \(\A\)-module.

\begin{prop}
     \label{prop:sympower is free}
   Let \(\mathbf{A}\) be a commutative $\OI$-algebra, and let \(\F\) be a finitely-generated free \(\A\)-module. Then, for every integer \(q \ge 0\), the symmetric power $\operatorname{Sym}_{q}(\F)$ is also a finitely-generated free \(\A\)-module.
\end{prop}

\begin{proof}
For any $\A$-modules $\F$ and $\mathbf{G}$, one has a decomposition 
\[
\operatorname{Sym}_{q}(\F \oplus \mathbf{G}) = \bigoplus\limits_{i+j = q} \operatorname{Sym}_{i}(\F) \otimes \operatorname{Sym}_{j}(\mathbf{G})
\]
This follows as in the proof of \Cref{powersum}. 
Thus, it suffices to show the statement in the case where \(\F = \F^{\OI,d}\). Now one argues as in the proof of \Cref{exteriorfree}. 
\end{proof}

To generalize the Buchsbaum-Eisenbud complexes of \Cref{sec:classical BE} to the \(\A\)-module setting, we need to define the dual of a free module. If \(\mathbf{M}\) is an \(\A\)-module, we have already seen that \(\operatorname{Nat}_{\A}(\mathbf{M}, \A)\) is not a suitable dual, because this is in general not an \(\A\)-module. Taking the width-wise dual of \(\mathbf{M}\) would produce a \textit{contravariant} functor from \(\OI\) to \(k\)-mod, which is not even an \(\OI\)-module. The natural candidate for the dual of an \(\A\)-module is obtained by using the internal hom defined in \cref{internalhom}.

\begin{defn}
  For an \(\A\)-module \(\mathbf{M}\), we define the \textit{dual of \(\mathbf{M}\)} to be \(\operatorname{Hom}_{\A}(\mathbf{M},\mathbf{A})\), and we denote it by \(\mathbf{M}^{*}\).
\end{defn}

Thus, in width $w$, we have $\M^* (w) = \operatorname{Nat}_{\A}(\mathbf{F}^{\OI,w}_{\A}\otimes_{\A} \mathbf{M}, \mathbf{A})$. 

Unfortunately, many statements about the dual in the category of modules over a commutative ring \(R\) do not generalize to the setting of modules over an \(\OI\)-algebra \(\A\). For example, if $d \ge 1$ and \(\mathbf{F} = \Fo{d}\), then $\mathbf{F} (0)$ is zero, but \(\mathbf{F}^{*}(0) =  \operatorname{Nat}_{\A}( \Fo{d}, \A)\) is not zero, and so $\Fo{d}$ is not isomorphic to its dual. However, free modules generated in width zero are isomorphic to their duals.

\begin{prop}
      \label{prop:dual is isomorphic}
Let \(\mathbf{F}\) be a finitely generated free \(\A\)-module generated in width \(0\), that is, $\mathbf{F} \cong (\Fo{0})^n$ for some 
   $n \in \N$.  Then there is an isomorphism \(\mathbf{F} \rightarrow \mathbf{F}^{*}\) of $\A$-modules. 
\end{prop}

\begin{proof}
  \begin{sloppypar}
  Note that \(\mathbf{F}^{\OI,0}_{\A}\) is isomorphic to \(\A\) as an \(\A\)-module.  Assume $\mathbf{F} \cong (\Fo{0})^n$ and let  \(\{f_1, f_2, \ldots, f_n\} \subset \mathbf{F} (0)\) be a  basis for \(\mathbf{F}\). If we let \(e_{\text{id}_w}\) be the generator of 
 \(\mathbf{F}^{\OI,w}_{\A}\), then \(\F^{\OI,w}_{\A} \otimes_{\A} \F\) is a free $\A$-module of rank \(n\).   Due to \Cref{tensorfree}, it has a basis consisting of width $w$ elements
 \[\{e_{\text{id}_w}\otimes \eta_* (f_1), e_{\text{id}_w}\otimes \eta_* (f_2), \ldots, e_{\text{id}_w}\otimes \eta_* (f_n)\}, 
 \] 
 where $\eta$ is the unique \(\OI\) morphism in $\Hom_{\OI} (0, w)$. Hence, any $\A$-module morphism 
 \mbox{\(\ffi \colon \Fo{w} \otimes_{\A} \F \to \A\)} is determined by the images of $e_{\text{id}_w}\otimes \eta_* (f_j)$ in $\A (w)$, say $a_j$.
\end{sloppypar}

For $i \in [n]$, define an $\A$-module morphism $f_i^* \colon  \Fo{0} \otimes_{\A} \F \to \A$ by mapping $e_{\text{id}_0}\otimes f_j$ onto $1$ if $j = i$ and onto $0$ if $j \neq i$. Using $\varepsilon \in \Hom_{\OI} (0, w)$ and the definition of the $\A$-module structure of $\F^* = \Hom_{\A} (\F, \A)$, a computation shows for the above morphism $\ffi$ that 
$\ffi = a_1 \F^* (\varepsilon) (f_1^*) + \cdots + a_n \F^* (\varepsilon) (f_n^*)$. It follows that the $\A$-module morphism $\F \to \F^*$ defined by \(f_i \mapsto f_i^*\) is an isomorphism.  
\end{proof}

\section{The OI Koszul Complex}
\label{sec:OI-Koszul}

\itodo{Put Thm 1.1 here} 

If $R$ is a commutative ring and $F$ is a finitely generated free $R$-module, then the Koszul complex to an $R$-module homomorphism $\ffi \colon F \to R$ is a bounded complex of finitely generated free $R$-modules that is generically acyclic and, if it is acyclic, resolves $\coker \ffi$. After the preceding  preparations, we establish in this section an analogous result for modules over a commutative  $\OI$-algebra $\A$.

\begin{thm} 
   \label{prop:Koszul}
Let $\A$ be a commutative $\OI$-algebra and consider a morphism of $\A$-modules $\ffi \colon \F \to \A$, where $\F$ is a finitely generated free $\A$-module. Then there is a complex of finitely generated free $\A$-modules 
\[
\cdots \longrightarrow \bigwedge\nolimits^{3}\mathbf{F} \longrightarrow \bigwedge\nolimits^{2}\mathbf{F} \longrightarrow \bigwedge\nolimits^{1}\mathbf{F} \longrightarrow \bigwedge\nolimits^{0}\mathbf{F} = \A \longrightarrow 0,
\] 
where each differential \(\partial_{i} \colon \bigwedge\nolimits^{i}\mathbf{F} \longrightarrow \bigwedge\nolimits^{i-1}\mathbf{F} \) is defined width-wise by letting 
\(\partial_{i}(w) \colon \bigwedge\nolimits^{i} (\mathbf{F}(w)) \longrightarrow \bigwedge\nolimits^{i-1} (\mathbf{F}(w))\)  be the degree 
\(i\) map of the classical Koszul complex to the $\A(w)$-module homomorphism $\ffi (w) \colon \mathbf{F}(w) \rightarrow \mathbf{A}(w)$. This complex is generically acyclic and, if it is acyclic, resolves $\coker \ffi$. 

Furthermore, if $\ffi$ is a morphism of graded $\A$-modules, then the above complex is a graded complex. 
\end{thm} 

We refer to the above complex as the \textit{$\OI$ Koszul complex} to $\ffi$ and denote it by $K_{\bullet} (\ffi)$. Note that the above result covers \Cref{intro:thm1} of the Introduction. 

\begin{proof}[Proof of \Cref{prop:Koszul}]
By \Cref{exteriorfree}, we know that any module $\bigwedge^i \F$ is finitely generated and free and that, for any $w \in \N_0$, one has $\big ( \bigwedge^i \F \big ) (w) = \bigwedge^i (\F(w))$. 
 Moreover, since $\partial_i (w)$ is the differential of a classical Koszul complex, the stated sequence is indeed a complex. Analogously, it follows that $K_{\bullet} (\ffi)$ is graded if $\ffi$ is graded. 

It remains to show that the family of $\A(w)$-module homomorphisms $\partial_i (w)$ can be used to define a morphism of $\A$-modules $\partial_i \colon \bigwedge\nolimits^{i}\mathbf{F} \longrightarrow \bigwedge\nolimits^{i-1}\mathbf{F}$, that is, for any $\eps \in \Hom_{\OI} (w, w')$, the diagram 
 \[
    \begin{tikzcd}
      \big ( \bigwedge^i \F \big ) (w) \ar[r,"\partial_i (w)"] \ar[d,"(\bigwedge^i \F)(\eps)"] & \big ( \bigwedge^{i-1} \F \big ) (w) \ar[d,"(\bigwedge^{i-1} \F)(\eps)"] \\
      \big ( \bigwedge^i \F \big ) (w') \ar[r,"\partial_i (w')"] & \big ( \bigwedge^{i-1} \F \big ) (w')
    \end{tikzcd}
  \]
commutes. This is a straightforward computation. We leave the details to the reader. 
\end{proof}

Since the modules $\bigwedge^i \F$ are free, there is an alternative way to describe the above $\OI$ Koszul complex. 
One can define the differential $\partial_i$ directly as a morphism of $\A$-modules by specifying the images of the basis elements of $\bigwedge^i \F$ and then arguing that in every  width $w$, the map $\partial_i (w)$ is the differential in a classical complex of $\A(w)$-modules. Note that $\OI$ Koszul complexes are infinite, in contrast to the classical Koszul complex.

Special cases of the above Koszul complex have appeared previously in the literature. 

\begin{example} 
   \label{exa:former Koszul}
The first $\OI$ Koszul complex was constructed in \cite{NR2}. Consider any $\A$-module morphism $\ffi \colon \Fo{1}_{\A} \to \A$. 
It is determined by $a = \ffi (e_{\id}) \in \A (1)$. Moreover, one computes $\bigwedge^i \Fo{1}_{\A} \cong \Fo{i}_{\A}$, and so the complex $K_{\bullet} (\ffi)$ has a particular simple form
\[
\cdots \longrightarrow \Fo{3}_{\A}  \longrightarrow \Fo{2}_{\A}  \longrightarrow \Fo{1}_{\A}  \longrightarrow \A \longrightarrow 0.  
\]
This is the version given in \cite[Lemma 3.8]{NR2}. In the case where, using notation from \cite{NR2}, $\A$ is the polynomial $\OI$-algebra $\A = \mathbf{X}^{\OI, 1} \cong \operatorname{Sym}_{\bullet} (\Fo{1})$, and so $\A (w) =  k [x_{j} \mid  j \in [w]]$ is a polynomial ring in $w $ variables, the Koszul complex to $\ffi \colon \Fo{1}_{\A} \to \A$ with $\ffi (e_{\id}) = x_1$ is acyclic and has the structure of a cellular resolution by \cite[Theorem 4.10]{mincell}. It resolves $\coker \ffi = \A/\langle x_1 \rangle$, and in fact gives a minimal graded free resolution of $\coker \ffi$. The existence of minimal resolutions in certain cases has been established in \cite[Theorem 3.10]{mincell}. They are unique up to isomorphism. 

Moreover, $K_{\bullet} (\ffi)$ is a width-wise minimal free resolution in the sense of \cite{mincell}, that is, in each width $w$, the complex $(K_{\bullet} (\ffi)) (w)$ is a minimal graded free resolution of $(\coker \ffi) (w)$, which is $ k[x_1,\ldots,x_w]/\langle x_1,\ldots,x_w \rangle \cong k$. 
\end{example}

It is worth mentioning the following generalization of the previous example. 

\begin{example} 
    \label{exa:Koszul non-noeth}
(i) For any integer $d \ge 1$, consider the Koszul complex to the $\OI$-morphism $\ffi \colon \Fo{d}_{\A} \to \A$. It is determined by $a = \ffi (e_{\id_d}) \in \A (d)$. Again, there are cases in which $K_{\bullet} (\ffi)$ is acyclic. For example, using notation from \cite{NR2}, let $\A$ be the polynomial $\OI$-algebra $\A = \mathbf{X}^{\OI, d} \cong \operatorname{Sym}_{\bullet} (\Fo{d})$, and so $\A (w) = k [x_{\pi} \mid \pi \in \Hom_{\OI} (d, w)]$ is a polynomial ring in $\binom{w}{d}$ variables. Then the complex $K_{\bullet} (\ffi)$ to $\ffi \colon \Fo{d}_{\A} \to \A$ with $\ffi (e_{\id_d}) = x_{\id_w}$ is acyclic. In fact, it gives a minimal and width-wise minimal graded free resolution of $\coker \ffi$. Note that $\A= \mathbf{X}^{\OI, d}$ is not a noetherian $\OI$-algebra if $d \ge 2$ (see \cite[Proposition 4.8]{NR2}), and so the existence of free resolutions with finitely generated $\A$-modules is not guaranteed in general. 

(ii) 
Analogous results are true over any polynomial $\OI$-algebra. For example, consider $A = \mathbf{X}^{\OI, 2} \otimes_k \mathbf{X}^{\OI, 3}$ and so $\A (w) = k [x_{\pi}, y_{\sigma} \mid \pi \in \Hom_{\OI} (2, w), \sigma \in \Hom_{\OI} (3, w)]$ is a poynomial ring in $\binom{w}{2} + \binom{w}{3}$ variables. Then the complex $K_{\bullet} (\ffi)$ to $\ffi \colon \Fo{2}_{\A} \oplus \Fo{3}_{\A}  \to \A$ with $\ffi (e_{\id_2}) = x_{\id_2}$ and $\ffi (e_{\id_3}) = y_{\id_3}$  is acyclic and gives a minimal and width-wise minimal free resolution of $\coker \ffi$. 
\end{example}

\begin{example}
As in the classical case, acyclicity of $K_{\bullet} (\ffi)$ depends on $\A$ and the choice of the morphism $\ffi \colon \F \to \A$. 
For example, let $\A = \mathbf{X}^{\OI, 1} \cong \operatorname{Sym}_{\bullet} (\Fo{1})$ as in \Cref{exa:former Koszul} and consider the morphism $\ffi \colon \Fo{2}_{\A} \to \A$ with $\ffi (e_{\id}) = x_2 \in \A(2)$. Then the Koszul complex $K_{\bullet} (\ffi)$ is not acyclic. For example,  its width 3 component is not acylic because it is the Koszul complex on the sequence $(x_2, x_3, x_3)$. 
%
\end{example}


\section{OI Buchsbaum-Eisenbud Complexes}
\label{sec:OI-BE}

Over a commutative ring $R$, the classical Buchsbaum-Eisenbud complexes are determined by an $R$-module homomorphism $\ffi \colon F \to G$ of finitely generated free $R$-modules (see \Cref{def:classical BE complex}). 
Similar to the Koszul complex in the previous section, we now generalize the Buchsbaum-Eisenbud complexes  to the \(\OI\) setting. 

\begin{thm}
\label{thm:OI BE}
Let $\A$ be a commutative $\OI$-algebra and consider a morphism of finitely generated free $\A$-modules $\ffi \colon \F \to \bG$, where $\bG$ is generated in width zero and has rank $r$. Then, for every integer $i \ge 0$, there is a complex of finitely generated free $\A$-modules 
\begin{align*}
 \cdots & \to \bigwedge\nolimits^{r}(\bG^{*})\otimes S_{p+1}(\bG^{*})\otimes \bigwedge \nolimits^{r+i+p+1} \F \to \bigwedge\nolimits^{r}(\bG^{*})\otimes S_{p}(\bG^{*})\otimes \bigwedge \nolimits^{r+i+p} \F \to  \cdots \\
 &  \to \bigwedge\nolimits^{r}(\bG^{*})\otimes S_{0}(\bG^{*})\otimes \bigwedge \nolimits^{r+i} \F  \stackrel{\alpha_i}{\longrightarrow}  \bigwedge\nolimits^{i} \F    \otimes S_{0}(\bG) \to  \bigwedge\nolimits^{i-1} \F \otimes S_{1}(\bG) \to \cdots \\
 & \to \bigwedge\nolimits^{1} \F \otimes S_{i-1}(\bG) \to \bigwedge\nolimits^{0} \F \otimes S_{i}(\bG) \to 0,  
 \end{align*}
 where $S_i$ is short for $\operatorname{Sym}_i$ 
and each differential \(\partial_{j}\) is defined width-wise by letting 
\(\partial_{j}(w)\)  be the degree 
\(j\) map of the classical Buchsbaum-Eisenbud  complex $\BE^i_{\bullet} (\ffi (w))$ to the $\A(w)$-module homomorphism $\ffi (w) \colon \mathbf{\F}(w) \rightarrow \mathbf{G}(w)$. This complex is generically acyclic and, if it is acyclic, resolves $\coker \partial_1$. 

Furthermore, if $\ffi$ is a morphism of graded $\A$-modules, then the above complex is a graded complex. 
\end{thm}

We refer to the above complex as the \textit{$\OI$ Buchsbaum-Eisenbud complex} to $\ffi$ and denote it by $\BE^{i}_{\bullet} (\ffi)$. Note that the above result covers \Cref{intro:thm2} of the Introduction. The assumption that $\bG$ is generated in width zero guarantees that the rank of $\bG (w)$ as an $\A(w)$-module is equal to the rank of $\bG$ for every width $w$. 

\begin{proof}[Proof of \Cref{thm:OI BE}]
  \begin{sloppypar}
By Propositions \ref{tensorfree},  \ref{exteriorfree} \ref{prop:sympower is free}, we know that each of the modules 
\mbox{$S_{p}(\bG^{*})\otimes_{\A} \bigwedge \nolimits^{q} \F $} is finitely generated and free and satisfies \[\big ( S_{p}(\bG^{*})\otimes_{\A} \bigwedge \nolimits^{q} \F ) (w) = S_{p}(\bG^{*}(w))   \otimes_{\A(w)}  \bigwedge \nolimits^{q} (\F(w))\] in every width $w$. These properties remain true if one tensors with $\bigwedge^r (\bG^*)$. 
 Moreover, since $\partial_j (w)$ is the differential of a classical Buchsbaum-Eisenbud complex $\BE^{i}_{\bullet} (\ffi (w))$, the stated sequence is indeed a complex. Analogously, it follows that $\BE^{i}_{\bullet} (\ffi)$ is graded if $\ffi$ is graded. 
\end{sloppypar}

It remains to show that the family of $\A(w)$-module homomorphisms $\partial_j (w)$ can be used to define a morphism of $\A$-modules $\partial_j$, which requires showing that certain diagrams commute. In $\BE^{i}_{\bullet} (\ffi)$, there are three types of differentials $\partial_j$ depending on $j = i+1$, $j \le i$ or $j \ge i+2$. We discuss the needed commutativity in the most 
complex case, where $j = i+1$, that is, if the differential is the splicing map $\alpha_i$,  and leave the other cases to the interested reader. 
 
We have to show for any map $\eps \in \Hom_{\OI} (w, w')$ that the following diagram commutes: 
  \[
    \begin{tikzcd}
     \left(\bigwedge\nolimits^{r}(\mathbf{G}^{*})\otimes S_{0}(\mathbf{G}^{*})\otimes \bigwedge \nolimits^{r+i}(\mathbf{F})\right)(w)   \ar[r,"\alpha_i (w)"] \ar[d,"\eps_*"] & \left(\bigwedge\nolimits^{i}(\mathbf{F}) \otimes S_{0}(\mathbf{G})\right)(w)  \ar[d,"\eps_*"] \\
     \left(\bigwedge\nolimits^{r}(\mathbf{G}^{*})\otimes S_{0}(\mathbf{G}^{*})\otimes \bigwedge \nolimits^{r+i}(\mathbf{F})\right)(w')  \ar[r,"\alpha_i (w')"] & \left(\bigwedge\nolimits^{i}(\mathbf{F}) \otimes S_{0}(\mathbf{G})\right) (w')
    \end{tikzcd}
  \]
To this end fix bases $\{g_1,\ldots,g_r\}$ and $\{f_1,\ldots,f_s\}$ of $\bG$ and $\F$, respectively. Thus, for each $v \in \N_0$, the module $\bG (v)$ has a basis  \(\{\tau_*(g_1), \ldots, \tau_*(g_r)\}\), where $\tau$ is the unique \(\OI\) morphism from width \(0\) to 
width \(v\). To simplify notation, we still use $g_j$ to denote $\tau_* (g_j)$. Similarly, we write $g_j^*$ for the basis elements of $\bG^*$ and $\bG^* (v)$ as introduced in the proof of  \Cref{prop:dual is isomorphic}. The elements $\pi_* (f_k)$ with $k \in [s]$ and $\pi \in \Hom_{\OI} (w_k, v)$, where $w_k$ is the width of $f_k$, form a basis of $\F (v)$. We order these elements by using the following order of the set 
$\{ (\pi, k) \mid k \in [s], \pi \in \Hom_{\OI} (w_k, v) \}
$: 
Put $(\pi, j) > (\tau, k)$ if $j < k$ or if $j = k$ and $\im \pi > \im \tau$ in the lexicographic order of the set $\N^{w_k}$. Observe that $(\pi, j) > (\tau, k)$ implies $(\rho \circ \pi, j) > (\rho \circ \tau, k)$  for every $\rho \in \Hom_{\OI} (v, v')$.  

Using this order, lets us rewrite the basis of $\F (w)$ as 
$\{f_{1},\ldots,f_{s'}\}$ where $s'$ is the rank of $\F (w)$. For a subset $J = \{j_1,\ldots,j_{i+r}\}$ of $[s']$ with cardinality $i+r$, set 
$f_J = f_{j_1} \wedge \cdots \wedge f_{j_{i+r}}$. Denote by $\eps (J)$ the set $\{k_1,\ldots,k_{i+r}\}$, where $k_m$ is the index of $\eps_* (f_{j_m})$ in the ordered basis $\{f'_1,\ldots,f'_{s''}\}$ of $\F(w')$. Thus, $\eps_* (f_J) = f'_{\eps (J)}$. 
Abusing notation, write $\ffi (v)$ for the coordinate matrix of $\ffi (v)$ with respect to the above ordered bases, and let $\ffi (v)_I$ be the submatrix formed by the columns with indices in $I$. Now we compute, where we use the facts on the order of basis elements mentioned above and the sign $\operatorname{sgn}(I\subseteq J)$ as introduced in  \Cref{def:classical BE complex}: 
\begin{align*}
      \big ( \varepsilon_* \circ \alpha_i  \big ) \left(  g_1^* \wedge \ldots \wedge g_r^* \otimes f_J \right) 
      &= \varepsilon_* \Big( \sum\limits_{\substack{I \subseteq J\\ |I| = r}} \operatorname{sgn}(I \subseteq J_w) \cdot \operatorname{det}(\varphi(w)_I) f_{J \setminus I}\Big)\\
      &= \sum\limits_{\substack{I \subseteq J \\|I| = r}} \operatorname{sgn}(\varepsilon(I) \subseteq \eps(J))\cdot \operatorname{det}\left(\varphi(w')_{\varepsilon(I)}\right) f'_{\eps(J) \setminus \varepsilon(I)})\\
      &= \alpha_i (g_1^* \wedge \ldots \wedge g_r^* \otimes f'_{\eps(J)})\\
      &= (\alpha_i \circ \varepsilon_*) \left(  g_1^* \wedge \ldots \wedge g_r^* \otimes   f_J \right). 
    \end{align*}
This proves the desired commutativity of the above diagram.     
\end{proof}

We illustrate \Cref{thm:OI BE} with a few examples, in which we focus on acyclic complexes that provide minimal free resolutions in every width $w$, that is, $\BE^i_{\bullet} (\ffi)$ gives a width-wise minimal free resolution over the algebra $\A$. 

\begin{example}
   \label{exa:BE noeth}
For any integer $c \ge 1$, let $\A$ be the graded polynomial $\OI$-algebra 
$(\mathbf{X}^{\OI, 1})^{\otimes c} \cong \operatorname{Sym}_{\bullet} ((\Fo{1})^{c})$, and so $\A (w) =  k [x_{i, j} \mid i \in [c],   j \in [w]]$ is a standard-graded polynomial ring in $c \cdot w$ variables. Consider the graded $\A$-module morphism $\ffi \colon \F = \Fo{1}_{\A} (-1) \to \A^c$ determined by $\ffi (e_{\id}) = \begin{bmatrix}
x_{1, 1} \\
\vdots \\
x_{c, 1}
\end{bmatrix}$. Thus, using the above standard bases, the coordinate matrix of  $\ffi(w)$ is a generic $c \times w$ matrix $(x_{i, j})$. For every $i \ge 0$, the complex $\BE^{i}_{\bullet} (\ffi)$ is acyclic and gives a width-wise minimal free resolution of $\partial_1$. 
For example, $\BE^0_{\bullet} (\ffi)$ is in every witdth $w \ge c$ the classical Eagon-Northcott complex that resolves $\A(w)/\bI (w)$, where $\bI (w)$ is the ideal generated by the maximal minors of the generic $c \times w$ matrix. If $0 \le w < c$ then $\BE^0_{\bullet} (\ffi) (w)$ has length zero, and the cokernel of $\partial_1 = \alpha_0$ is $\A(w)$. 

Note that for $c=3$, this covers the resolution discussed in \Cref{runningexample:summary}. 
\end{example}

A similar construction gives width-wise minimal free resolutions over non-noetherian polynomial algebras. We use the notation of \Cref{exa:Koszul non-noeth}

\begin{example}
   \label{exa:BE non-noeth}
(i) 
For any integers $c, d \ge 1$, consider the Buchsbaum-Eisenbud complexes to a graded $\OI$-morphism  $\ffi \colon \Fo{d} (-1) \to \A^c$, where $\A = (\mathbf{X}^{\OI, d})^{\otimes c} \cong \operatorname{Sym}_{\bullet} ((\Fo{d})^{c})$ and $\ffi$ is determined by $\ffi (e_{\id_d}) = \begin{bmatrix}
x_{1, \id_d} \\
\vdots \\
x_{c, \id_d}
\end{bmatrix}$. Thus, a coordinate matrix of $\ffi(w)$ is given by the generic $c \times \binom{w}{d}$ matrix $(x_{i, \pi})$ with $i \in [c]$ and $\pi \in \Hom_{\OI} (d, w)$. Let $\bI$ the ideal of $\A$ such that $\bI (w) \subset \A (w)$ is the ideal generated by the maximal 
minors of $\ffi (w)$. Considering $\BE^0_{\bullet} (\ffi)$, note that $\bI = \im \alpha_0$ and that $\BE^0_{\bullet} (\ffi)$ is a width-wise minimal 
graded free resolution of $\bI$. If $d = 1$, then $\bI$ is the ideal generated in width $c$ by the determinant of $\ffi (c)$. However, if $d \ge 2$ then the minimal generators of $\bI$ have different widths. In fact, their widths are given by the widths of the basis 
elements of $\bigwedge^c (\Fo{d}_{\A})$. By \Cref{exteriorfree}, it follows that $\bI$ has minimal generators in  any width $w$ with $w_0 \le w \le d c$, where $w_0 = \min \{ w \in \N_0 \mid\binom{w}{d} \ge c\}$. 
Note, that in this case $\BE^i_{\bullet} (\ffi)$ gives a width-wise minimal 
graded free resolution of $\coker \partial_1$ for every $i \ge 0$. 

(ii) 
Similar results are true over any polynomial $\OI$-algebra. For example, consider the graded algebra 
$A = \big ( \mathbf{X}^{\OI, 2} \otimes_k \mathbf{X}^{\OI, 3} \big )$. Then, for every $i \ge 0$,  the Buchsbaum-Eisenbud  complex $\BE^i_{\bullet} (\ffi)$ to $\ffi \colon \Fo{2}_{\A} (1) \oplus \Fo{3}_{\A} (-1) \to \A^4$ with 
\[
\ffi (e_{\id_2}) = \begin{bmatrix}
x_{1, \id_2} \\
\vdots \\
x_{4, \id_2}
\end{bmatrix} 
\text{ and } ~
\ffi (e_{\id_3}) = \begin{bmatrix}
y_{1, \id_3} \\
\vdots \\
y_{4, \id_3}
\end{bmatrix} 
\]
is acyclic and gives a minimal and width-wise minimal graded free resolution of $\coker \partial_1$. In particular, $\BE^1_{\bullet} (\ffi)$ resolves $\coker \ffi$. 

\end{example}

%


\itodo{Could talk about regular sequences and depth as a ``future directions'' remark}


\bibliographystyle{abbrev}

\begin{thebibliography}{10} 

\bibitem{BH} 
W.~Bruns and  J.~Herzog,
\emph{Cohen-Macaulay rings},
Revised edition, Cambridge Studies Adv. Math.
{\bf 39}, University Press, Cambridge, 1998.

\bibitem{BV}
W.~ Bruns and U.~Vetter, \emph{Determinantal rings}, Lecture Notes in Mathematics {\bf 1327}, Springer-Verlag, Berlin, Heidelberg, 1988.

\bibitem{BE} 
D.~A.~Buchsbaum and D.~Eisenbud,  \emph{Generic free resolutions and a family of generically perfect ideals}, Adv. Math {\bf 18} (1975), 245--301. 

 \bibitem{CEF} T.~Church, J.~S. Ellenberg, and B.~Farb,
   \emph{\(\FI\)-modules and stability for representations of
     symmetric groups}, Duke Math.\ J.\ {\bf 164} (2015), 1833--1910.

\bibitem{DEKL}
J.~Draisma, R.H.~Eggermont, R.~Krone, and A.~Leykin, \emph{Noetherianity for infinite-dimensional toric varieties}, Algebra Number Theory {\bf 9} (2015), 1857--1880.  

\bibitem{EN} 
J.A.~Eagon and D.G.Northcott, \emph{Ideals defined by matrices and a certain complex associated with them}, Proc. Roy. Soc. London Ser. A {\bf 269} (1962), 188-204. 

 \bibitem{E-book} 
 D.~Eisenbud,  \emph{Commutative algebra},
   Graduate Texts in Mathematics {\bf 150},  Springer-Verlag, New York,1995.

\bibitem{mincell}
N.~Fieldsteel and U.~Nagel, \emph{Minimal and cellular free resolutions over polynomial OI-algebras}, Preprint, 2021; available at  arXiv:2105.08603. 

\bibitem{HS}
C.J.~Hillar and S.~Sullivant,
\emph{Finite Gr\"obner bases in infinite dimensional polynomial rings and applications},
Adv.\ Math. {\bf 229} (2012), 1--25. 

 \bibitem{NR1} 
 U.~Nagel and T.~R\"omer, \emph{Equivariant Hilbert series in non-Noetherian Polynomial Rings}, 
J.\ Algebra {\bf 486} (2017), 204--245. 


 \bibitem{NR2} U.~Nagel and T.~R\"omer, \emph{\(\FI\)- and
     \(\OI\)-modules with varying coefficients}, J. Algebra {\bf 535}
   (2019), 286--322.
   
 \bibitem{N-hilb} U.~Nagel, \emph{Rationality of Equivariant Hilbert
     Series and Asymptotic Properties}, Trans.\ Amer.\ Math.\ Soc.\ {\bf 374} (2021), 7313--7357. 
     
\bibitem{LNNR}
D.V. Le, U. Nagel, H.D. Nguyen, and T. R\"{o}mer,
\emph{Castelnuovo-Mumford regularity up to symmetry},      Int.\ Math.\ Res.\ Not.\ \ {\bf 2021}, no.\ 14 (2021), 11010--11049.

   
 \bibitem{SS-17} 
 S.~V.~Sam and A.~Snowden, \emph{Gr\"obner methods
     for representations of combinatorial categories}, J.\ Amer.\
   Math.\ Soc.\ {\bf 30} (2017), 159--203.   
  
\end{thebibliography}

\end{document}